\documentclass[11pt]{amsart}
\usepackage{amssymb}
\usepackage{amsthm}
\usepackage{verbatim}

\theoremstyle{plain}
\newtheorem{prop}{Proposition}

\newtheorem{thm}{Theorem}

\newtheorem{lemma}{Lemma}

\def\order{{\mathrm{order}}}

\def\TSG{{\mathrm{TSG}}}
\def\Aut{{\mathrm{Aut}}}
\def\Diff{{\mathrm{Diff}}}

\def\int{{\mathrm{int}}}

\def\cl{{\mathrm{Cl}}}

\begin{document}

 \title{Topological Symmetry Groups of Graphs in $3$-Manifolds}
\author{Erica Flapan and Harry Tamvakis}

\thanks{The first author was supported in part by
NSF Grant DMS-0905087.}

\thanks{The second author was supported in part by
NSF Grant DMS-0901341.}

\address{Department of Mathematics, Pomona College,
Claremont, CA 91711, USA}

\email{eflapan@pomona.edu}

\address{Department of Mathematics, University of Maryland, College
Park, MD 20742, USA}

\email{harryt@math.umd.edu}

\date \today
 \maketitle
\begin{abstract} 
We prove that for every closed, connected, orientable, irreducible
3-manifold, there exists an alternating group $A_n$ which is not the
topological symmetry group of any graph embedded in the manifold. We
also show that for every finite group $G$, there is an embedding $\Gamma$ 
of some graph in a hyperbolic rational homology 3-sphere such that the
topological symmetry group of $\Gamma$ is isomorphic to $G$.
\end{abstract}

\section{Introduction} 
 Characterizing the symmetries of a molecule is an important step in
 predicting its behavior.  Chemists often model a molecule as a graph
 in $\mathbb{R}^3$.  They define the {\it point group} of a molecule
 as the group of isometries of $\mathbb{R}^3$ which take the molecular
 graph to itself.  This is a useful way of representing the symmetries
 of a rigid molecule.  However, molecules which are flexible or
 partially flexible may have symmetries which are not induced by
 isometries of $\mathbb{R}^3$. Jon Simon \cite {Si} introduced the
 concept of the {\it topological symmetry group} in order to study the
 symmetries of such non-rigid molecules.  This group has not only been
 used to study the symmetries of such molecules, but also the
 non-rigid symmetries of any graph embedded in $S^3$.  In this paper
 we extend this study to graphs embedded in other 3-manifolds.

In 1938, Frucht \cite{Fr} showed that every finite group is the
automorphism group of some connected graph.  By contrast, in
\cite{FNPT} we proved that the only finite simple groups which can
occur as the topological symmetry group of a graph embedded in $S^3$
are cyclic groups and the alternating group $A_5$.  Thus, we can say
that automorphism groups of connected graphs are {\it universal} for
finite groups, while topological symmetry groups of graphs embedded in
$S^3$ are not.  We now prove that topological symmetry groups of
graphs embedded in any given closed, connected, orientable,
irreducible $3$-manifold are not universal for finite groups.

In particular, let $M$ be a $3$-manifold, let $\gamma$ be an abstract
graph, and let $\Gamma$ be an embedding of $\gamma$ in $M$.  Note that
by a {\it graph} we shall mean a finite, connected graph with at
most one edge between any pair of vertices and two distinct vertices
for every edge. The {\it topological symmetry group} $\TSG(\Gamma, M)$
is defined to be the subgroup of the automorphism group $\Aut(\gamma)$
consisting of those automorphisms of $\gamma$ which are induced by a
homeomorphism of the pair $(M,\Gamma)$.  Allowing only orientation
preserving homeomorphisms defines the {\it orientation preserving
  topological symmetry group} $\TSG_+(\Gamma,M)$.  We prove the
following result.

\begin{thm}  
 \label{mainthm}
For every closed, connected, orientable, irreducible 3-manifold $M$,
there exists an alternating group $A_n$ which is not isomorphic to
$\TSG(\Gamma,M)$ for any graph $\Gamma$ embedded in $M$.
\end{thm}

On the other hand, it follows from Frucht's theorem \cite{Fr}
that every finite group $G$ can occur as the topological symmetry
group of a graph embedded in some closed 3-manifold (which depends on
$G$).  We prove the following stronger result.

\begin{thm}  
\label{univthm}
For every finite group $G$, there is an embedding $\Gamma$ of a graph
in a hyperbolic rational homology 3-sphere $M$ such that 
$\mathrm{TSG}(\Gamma,M) \cong G$.
\end{thm}
  
   A graph is said to be {\it 3-connected} if at least 3
   vertices must be removed together with the edges they contain in
   order to disconnect the graph or reduce it to a single vertex.  In
   \cite{FNPT}, we showed that for every embedded 3-connected graph
   $\Gamma$ embedded in $S^3$, there is a subgroup of
   $\Diff_+(S^3)$ isomorphic to $\TSG_+(\Gamma,S^3)$.  We now
   prove this is not the case for all 3-manifolds.
   
 \begin{thm}
 \label{rmk1}
For every closed, orientable, irreducible, $3$-manifold $M$ that is
not Seifert fibered, there is an embedding of a $3$-connected graph
$\Gamma$ in $M$ such that $\TSG_+(\Gamma,M)$ is not isomorphic to any
subgroup of $\Diff_+(M)$.
 \end{thm}

\section{Proof of Theorem \ref{mainthm}}  
 
 We begin by introducing some notation.  Let $\Gamma$ be an embedding
of a graph in a $3$-manifold $M$.  Let $V$ denote the set of embedded
vertices of $\Gamma $, and let $E$ denote the set of embedded edges of
$\Gamma $.  We construct a neighborhood $N(\Gamma )$ as the union of
two sets, $N(V)$ and $N(E)$, which have disjoint interiors. In particular, for each
vertex $v\in V$, let $N(v)$ denote a ball around $v$ whose
intersection with $\Gamma$ is a star around $v$, and let $N(V)$ denote
the union of all of these balls.  For each embedded edge $e \in E$,
let $N(e )$ denote a solid cylinder $D^{2}\times I$ whose core is
$e-N(V)$, such that $N(e )\cap \Gamma\subseteq e$, and $N(e)$ meets
$N(V)$ in a pair of disks.  Let $N(E)$ denote the union of all the
solid cylinders $N(e )$.  Let $N(\Gamma )=N(V)\cup N(E)$. such shall
use $\partial 'N(e )$ to denote the annulus $\partial N(\Gamma )\cap
N(e )$ in order to distinguish it from the sphere $\partial N(e)$.

By the standard smoothing results in dimension 3 (proved in
\cite{Moi}), if a particular automorphism of an embedded graph
$\Gamma$ can be induced by an orientation preserving homeomorphism of
$(M,\Gamma)$, then the same automorphism can be induced by an
orientation preserving homeomorphism of $(M,\Gamma)$ which is a
diffeomorphism except possibly on the set of vertices of $\Gamma$.
Thus we shall abuse notation and call such a homeomorphism a {\it
  diffeomorphism} of $(M, \Gamma)$.  
\smallskip

 \begin{lemma}
\label{ball}  
Let $\Gamma$ be a graph embedded in ball $B\subseteq M$ an
orientable, irreducible, $3$-manifold, and suppose that $\TSG_+(\Gamma,M)$
is a non-abelian simple group.  Then $\TSG_+(\Gamma,M)\cong A_5$.
\end{lemma}

\begin{proof}   
If $M=S^3$ the result follows from \cite{FNPT}.  Thus we assume
$M\not =S^3$.  We obtain an embedding of $\Gamma$ in $S^3$ by gluing a
ball to the outside of $B$.  Since $\Gamma$ is connected, $S^3-\Gamma$
is irreducible.  Thus $M-\Gamma$ is the connected sum of the
irreducible manifolds $M$ and $S^3-\Gamma$.  Hence the splitting
sphere $\partial B$ is unique up to isotopy in $M -\Gamma$.  Thus
$\TSG_+(\Gamma, M)$ is induced by a group $G$ of orientation
preserving diffeomorphisms of $(M,\Gamma)$ which takes $B$ to itself.

Since the restriction of $G$ to $B$ extends radially to $S^3$ taking
$\Gamma$ to itself, it follows that $\TSG_+(\Gamma,M) \leq
\TSG_+(\Gamma,S^3)$.  We show containment in the other direction as
follows.  Let $p$ be a point of $S^3$ which is disjoint from $\Gamma$.
Let $g$ be an orientation preserving diffeomorphism of the pair
$(S^3,\Gamma)$.  We can compose $g$ with an isotopy to obtain an
orientation preserving diffeomorphism $h$ of the pair $(S^3,\Gamma)$
which pointwise fixes a neighborhood of $p$ disjoint from $\Gamma$ and
induces the same automorphism of $\Gamma$ as $g$.  It follows that
$\TSG_+(\Gamma,S^3)$ can be induced by a group $G'$ of orientation
preserving diffeomorphisms of $(S^3,\Gamma)$ which take the ball $B$
to itself fixing its boundary pointwise.  We may therefore restrict
the elements of $G'$ to $B$ and then extend them to $M$ by the
identity.  It follows that $\TSG_+(\Gamma,M) = \TSG_+(\Gamma,S^3)$,
and hence by \cite{FNPT} we have $\TSG_+(\Gamma,M)\cong
A_5$.  \end{proof}

 \smallskip

\begin{lemma}
 \label{incompress}
 Let $\Gamma$ be a graph embedded in a closed, orientable, irreducible
 $3$-manifold $M$ such that $\TSG_+(\Gamma,M)$ is a non-abelian simple
 group.  Then there is a graph $\Lambda$ embedded in $M$ with
 $\TSG_+(\Lambda,M)\cong \TSG_+(\Gamma,M)$ such that $\partial
 N(\Lambda)-\partial'N(E)$ is incompressible in $\cl(M-N(\Lambda))$.
 \end{lemma}
 
\begin{proof}  The proof of this lemma is similar to the proof of a stronger result
\cite[Prop.\ 2]{FNPT} for $M=S^3$. We will therefore give only the
main ideas of the argument, with many of the details omitted.

We say $\Gamma$ has a {\it separating ball} $B$, if $B$
meets $\Gamma$ in a single vertex $v$ with valence at least 3 and $\mathrm{Int}(B)$ and $M-B$ each have non-empty intersection with
$\Gamma$.  In this case, we say $B\cap \Gamma$ is a {\it branch} of
$\Gamma$ at $v$.

Suppose that $\partial N(\Gamma)-\partial'N(E)$ is compressible in
$\cl(M-N(\Gamma))$.  Then there is a sphere in $M$ meeting $\Gamma$ in
a single point such that each complementary component of the sphere intersects
$\Gamma$ non-trivially.  Since $M$ is irreducible, one of these
components is ball.  Also, $\Gamma$ cannot be
homeomorphic to an arc because $\TSG_+(\Gamma,M)$ is non-abelian.
Thus $\Gamma$ has a separating ball at some vertex $v$.

First suppose, for the sake of contradiction, that $\TSG_+(\Gamma,M)$
fixes $v$ but does not setwise fix some branch $\Gamma_1$ at $v$.  Let
$\{\Gamma_1,\dots,\Gamma_n\}$ denote the orbit of $\Gamma_1$ under
$\TSG_+(\Gamma,M)$.  Now the action of $\TSG_+(\Gamma,M)$ on
$\{\Gamma_1,\dots,\Gamma_n\}$ defines a non-trivial monomorphism
$\Phi: \TSG_+(\Gamma,M)\rightarrow S_n$.  To see that $\Phi$ is onto,
let $(ij)$ be a transposition in $S_n$.  Since $\Gamma_i$ and
$\Gamma_j$ are in the orbit of $\Gamma_1$ under $\TSG_+(\Gamma,M)$, we
can choose a pair of separating balls $B_i$ and $B_j$ which are
disjoint except at $v$ such that there is an orientation preserving
diffeomorphism $g$ of $(M, \Gamma)$ with
$g((B_i,\Gamma_i))=(B_j,\Gamma_j)$.  We can then choose a separating
ball $E$ containing $B_i\cup B_j$ such that $\partial E\cap (B_i\cup
B_j)=\{v\}$.  Now define an orientation preserving diffeomorphism
$h:(E,\Gamma_i\cup \Gamma_j)\to E$ such that $h|B_i=g|B_i$, $h|B_j=g^{-1}|B_j$, and
$h|\partial E$ is the identity.  Finally, extend $h$ to $M-E$ by the
identity.  Thus we have $h:(M,\Gamma)\rightarrow (M,\Gamma)$ and $h$
induces $(ij)$ on $\Gamma$, and therefore $\Phi$ is onto.  Since
$\TSG_+(\Gamma,M)$ is simple and non-abelian, this is impossible.

Let $m$ denote an integer which is larger than the number of vertices
in $\Gamma$.  We will show that there is another graph $\Gamma'$
in $M$ such that $\TSG_+(\Gamma',M)\cong \TSG_+(\Gamma,M)$ and
$\Gamma'$ has fewer (possibly 0) branches at $v$.  By repeating this
argument as necessary one eventually obtains the embedded graph
$\Lambda$.  By the above paragraph, we only need to consider the
following two cases.

\smallskip

{\bf Case 1:}  $\TSG_+(\Gamma,M)$ fixes $v$ and setwise fixes every branch of $\Gamma$ at $v$.
\smallskip

Since $\TSG_+(\Gamma,M)$ is non-trivial, there is some branch $\Gamma_1$ at $v$ on which $\TSG_+(\Gamma,M)$ acts non-trivially.   Let $e_1$ denote an
edge in $\Gamma_1$ containing $v$, and let $\{e_1,\dots, e_r\}$ be
the orbit of $e_1$ under $\TSG_+(\Gamma,M)$.  We define $\Gamma'$ as
$\Gamma_1$ together with $m$ vertices of valence 2 added to the
interior of each $e_i$.  We can define a non-trivial monomorphism
$\psi: \TSG_+(\Gamma,M)\rightarrow \TSG_+(\Gamma',M)$ since each
automorphism of $\Gamma$ induces an automorphism of $\Gamma'$. Because
of the $m$ vertices on each $e_i$, every automorphism of $\Gamma'$
fixes $v$.  It is not hard to check that $\psi$ is onto, and $\Gamma'$
has fewer branches at $v$ than $\Gamma$.

\smallskip

{\bf Case 2:}  $v$ is not fixed by $\TSG_+(\Gamma,M)$.

\smallskip

Let $\{\Gamma_1,\dots,\Gamma_n\}$ denote the orbit of $\Gamma_1$ under
$\TSG_+(\Gamma,M)$.  Let $e_1$ denote an edge of the graph 
$\mathrm{cl}(\Gamma-(\Gamma_1\cup\dots\cup\Gamma_n))$ that contains
$v$, and let $\{e_1,\dots, e_r\}$ and $\{v_1,\dots, v_q\}$ denote the
orbits of $e_1$ and $v$ respectively under $\TSG_+(\Gamma,M)$. Now
define $\Gamma'$ as
$\mathrm{cl}(\Gamma-(\Gamma_1\cup\dots\cup\Gamma_n))$ together with
$m$ vertices of valence 2 added to the interior of each $e_i$.  We can
define a non-trivial monomorphism $\psi: \TSG_+(\Gamma,M)\rightarrow
\TSG_+(\Gamma',M)$ since each automorphism of $\Gamma$ induces an
automorphism of $\Gamma'$.  Since, $\TSG_+(\Gamma',M)$ cannot be cyclic, $\Gamma'$ cannot be a simple closed curve.  Thus every automorphism of $\Gamma'$ leaves $\{v_1,\dots, v_q\}$
setwise invariant.  Now it is not hard to check that $\psi$ is onto
and $\Gamma'$ has fewer branches at $v$ than $\Gamma$.\end{proof}
 
 \bigskip

In addition to Lemmas 1 and 2, our proof of the next result will 
employ Jaco--Shalen \cite{JS} and Johannson's \cite{Jo} theory of
characteristic splittings along tori and annuli.  For a survey of the
definitions and statements see \cite{Bo1}.
 
 \smallskip

 \begin{prop}  
 \label{mainprop}
Let $M$ be a closed, connected, orientable, irreducible $3$-manifold.
Then there exists an integer $d$ depending only on $M$ such that for any graph
$\Gamma$ embedded in $M$ if $\TSG_+(\Gamma,M)$ is a
non-abelian simple group, then the order of $\TSG_+(\Gamma,M)$ is at
most $d$.
\end{prop}
 
 \begin{proof}  
We shall assume that $\Gamma$ is not contained in a ball in $M$, since otherwise the result
follows from Lemma \ref{ball}.  Since the proof is lengthy we divide it
into six steps.
\smallskip

\noindent {\bf Step 1:} We pick the number $d$.

\smallskip

 Apply the Characteristic Submanifold Theorem \cite{JS,Jo} to $M$
to obtain a minimal collection $\Theta$ of incompressible tori in $M$
such that the closure of each component of $M-\Theta$ is either
atoroidal or Siefert fibered, and $\Theta$ is unique up to isotopy.
If $\Theta$ is non-empty, let $C_1,\dots, C_m$ denote the closures of
the components of $M-\Theta$, otherwise let $m=1$ and $C_1=M$.  Now we
consider two cases.
 
 \smallskip

\noindent{\bf Case 1:}  Either $\Theta$ is non-empty or $\Theta$ is empty and $M$ is
Seifert fibered.
\smallskip

  Consider a component $C_i$ which is Seifert fibered (possibly
$C_i=M$).  We see as follows that there is an upper bound $n_i$ on the
order of all finite simple non-abelian groups which can act faithfully
on the base space of the fibration of $C_i$.  If the base space is a
sphere or projective plane (possibly with holes), then $n_i=60$ is an
upper bound since $A_5$ is the only finite simple non-abelian group
which can act faithfully on a sphere or projective plane.  If the base
space has negative Euler characteristic $\chi$, then $n_i=84|\chi |$
is an upper bound by a classical result of Hurwitz \cite{Hu}.  Since
no finite simple non-abelian groups can act faithfully on a torus or
Klein bottle (possibly with holes), if the base space is one of these
surfaces we let $n_i=1$.
 
 Now consider a component $C_i$ which is not Seifert fibered.  By the
hypotheses of this case then $C_i\not =M$. It follows that $C_i$ is
atoroidal and has non-empty incompressible boundary.  Hence, by
Thurston's Hyperbolization Theorem \cite{Th}, $C_i$ admits a complete,
hyperbolic structure with totally geodesic boundary.  Now by Mostow's
Rigidity Theorem \cite{Mos}, there is an integer $n_i$ such that no
finite group of diffeomorphisms of $C_i$ has order greater than $n_i$.
 
After having chosen an $n_i$ associated with each component $C_i$, we let
$d=\mathrm{Max}\{n_1,\dots, n_m, m!, 60\}$. 

\smallskip

\noindent{\bf Case 2:} $\Theta$ is empty and $M$ is not Seifert fibered.
\smallskip

By the Geometrization Theorem \cite{MF, MT, MT2}, $M$ has a geometric
structure.  Also since $M$ is irreducible and not Seifert fibered, $M$
does not admit a circle action.  Furthermore, by the Elliptization
Theorem \cite{MT}, a 3-manifold with finite fundamental group is
elliptic and hence Seifert fibered.  Thus $M$ has infinite $\pi_1$, no
circle action, and is irreducible.  Hence by Kojima \cite{Ko} there is
a bound $q$ on the order of finite groups of diffeomorphisms of $M$.
In this case, we let $d=\mathrm{Max}\{q, 60\}$.

\smallskip

Now let $\Gamma$ be a graph embedded in $M$ such that
$\TSG_+(\Gamma,M)$ is a simple non-abelian group.  By Lemma
\ref{incompress}, without loss of generality we can assume that
$\partial N(\Gamma)-\partial'N(E)$ is incompressible in
$\cl(M-N(\Gamma))$.  We will prove that $\order(\TSG_+(\Gamma,M)) \leq
d$.

\smallskip

\noindent {\bf Step 2:} We choose $W\subseteq \cl(M-N(\Gamma))$ and a group $G$ of
diffeomorphisms  of $(M,\Gamma)$ leaving $W$ setwise invariant.

\smallskip

Since $\Gamma$ is connected and $M$ is irreducible, the manifold
$\cl(M-N(\Gamma))$ is irreducible.  Thus we can apply the
Characteristic Submanifold Theorem \cite{JS,Jo} to the manifold $\cl(M-N(\Gamma))$
to get a minimal family of characteristic tori.  When we split
$\cl(M-N(\Gamma))$ along these tori, $\partial N(\Gamma)$ is contained
entirely in one component, the closure of which we denote by $X$.
Since $\partial N(\Gamma)-\partial'N(E)$ is incompressible in
$\cl(M-N(\Gamma))$, $\Gamma$ has no valence one vertices.  Thus since
$\TSG_+(\Gamma,M)$ is not cyclic, it follows that $\Gamma$ contains at
least two simple closed curves.  Therefore the genus of $\partial
N(\Gamma)$ is at least two, and hence $X$ cannot be Seifert fibered.
It follows that $X$ is atoroidal, and since $M$ is irreducible, $X$ is
also irreducible.

  Let $P$ denote the set of annuli in $\partial'N(E)$ together with
the torus boundary components of $X$.  Since $\partial
N(\Gamma)-\partial'N(E)$ is incompressible in $\cl(M-N(\Gamma))$,
$\partial X-P$ is incompressible in $X$.  Thus we can now apply the
Characteristic Submanifold Theorem for Pared Manifolds \cite{JS,Jo} to
the pared manifold $(X,P)$.  Since $X$ is atoroidal, this gives us a
minimal family $\Omega$ of incompressible annuli in $X$ with
boundaries in $\partial X -P$ such that if $W$ is the closure of any
component of $X-\Omega$, then the pared manifold $(W, W\cap (P\cup
\Omega))$ is either simple, Seifert fibered, or $I$-fibered, and the
set $\Omega$ is unique up to isotopy.
  
  Let $G$ denote the collection of orientation preserving
diffeomorphisms of $(M,\Gamma)$ which leave $X$, $\Omega$, $N(V)$, and
$N(E)$ setwise invariant.  It follows from the uniqueness up to
isotopy of each of these sets that every automorphism in
$\TSG_+(\Gamma, M)$ is induced by some element of $G$.   Suppose that two elements of $G$
induce the same automorphism on $\Gamma$.  Then they induce the same
permutation on the set of components of $\partial N(V)$.  Since
$\Gamma$ has at most one edge between two vertices and every edge has
two distinct vertices, they also induce the same permutation on the
set of components of $\partial 'N(E)$.  It now follows that they
induce the same permutation on the set of annuli in $\Omega$ as well
as on the components of $X-\Omega$.

We construct a graph $\lambda$ associated with $(X,\Omega)$ by
representing the closure of each component of $X-\Omega$ by a vertex
and defining an edge between a pair of vertices if the components they
represent are adjacent. Observe that each annulus in $\Omega$ can be
capped off by a pair of disks in $N(V)$ to obtain a sphere in $M$.
Since $M$ is irreducible, any sphere in $M$ separates.  It follows
that all of the annuli in $\Omega$ separate $X$.  Thus $\lambda$ is a
tree.  Now $G$ induces a group of automorphisms on $\lambda$ which
either fixes a vertex or fixes an edge setwise.  Suppose that no
vertex of $\lambda$ is fixed by $G$.  Then some edge of $\lambda$ is
inverted by an element of $G$.  Hence there is an annulus $A\in
\Omega$ which is setwise invariant under $G$ and some element of $G$
which interchanges the components adjacent to $A$.  Thus we can define
a non-trivial homomorphism $\varphi:\TSG_+(\Gamma,M)\rightarrow
\mathbb{Z}_2$ where $\varphi(a)=1$ if and only if $a$ is induced by an
element of $G$ which interchanges the components adjacent to $A$.
However this is impossible since $\TSG_+(\Gamma,M)$ is simple and
non-abelian.  Hence there must be a vertex of $\lambda$ which is
fixed by $G$.
  
  If the action of $G$ on $\lambda$ is non-trivial, then we choose a
particular fixed vertex $w$ of $\lambda$ which is adjacent a vertex
which is not fixed by $G$.  In this case, let $W$ be the closure of
the component of $X-\Omega$ represented by $w$.  Then $W$ is setwise
fixed by $G$ and some annuli in $W\cap \Omega$ are not setwise fixed
by $G$.  Observe that all of the components of $W\cap \Omega$ are
contained in a single component of $\partial W$.  If this component of
$\partial W$ were a torus, then there would be a non-trivial
homomorphism from $\TSG_+(\Gamma,M)$ to $\mathbb{Z}_r$ where $r$ is
the number of annuli in $W\cap \Omega$.  Since $\TSG_+(\Gamma,M)$ is a
finite simple non-abelian group, this component of $\partial W$ cannot
be a torus.
  
  If the action of $G$ on $\lambda$ is trivial, we choose $W$ to be
  the closure of some component of $X-\Omega$ such that some
  components of $\partial N(V)\cap W$ are permuted by $G$.  We know
  there is such a component of $X-\Omega$ since $\TSG_+(\Gamma,M)$ is
  non-trivial and $\partial N(\Gamma)$ is contained in $X$.  Observe
  that since the annuli of $\Omega$ separate $X$, and $\partial
  N(\Gamma)$ is connected, all of the components of $W\cap \partial
  N(V)$ are contained in a single component of $\partial W$.  As
  above, if the component of $\partial W$ containing $W\cap \partial
  N(V)$ were a torus, then there would be a non-trivial homomorphism
  from $\TSG_+(\Gamma,M)$ to a finite cyclic group.  Thus again, this
  component $\partial W$ is not a torus.
  
   Since not all of the components of $\partial W$ are tori, $W$
cannot be Seifert fibered.  Thus the pared manifold $(W, W\cap (P\cup
\Omega))$ is either $I$-fibered or simple.
  
    \smallskip

\noindent {\bf Step 3:} We prove that the pared manifold $(W, W\cap
(P\cup \Omega))$ is simple.

\smallskip

 Suppose, for the sake of contradiction, that $(W, W\cap (P\cup
\Omega))$ is $I$-fibered.  Then there is an $I$-bundle map of $W$ over
a surface such that $W\cap (P\cup \Omega)$ is the preimage of the
boundary of the surface.  It follows that the corresponding $\partial
I$-bundle is $\partial N(V)\cap W$ and has either one or two
components.  If $\partial N(V)\cap W$ has two components which are
interchanged by some element of $G$, then there would be a non-trivial
homomorphism from $\TSG_+(\Gamma,M)$ to $\mathbb{Z}_2$. Thus we can
assume that each component of $\partial N(V)\cap W$ is setwise
invariant under $G$.  It now follows from our definition of $W$ that
some annulus $F_1$ in $ \Omega \cap W$ is not setwise invariant under
$G$.

Let $\{ F_1,\dots, F_r\}$ denote the orbit of $F_1$ under $G$.  Since
$r>1$, the boundary components of $F_1$ do not co-bound an annulus in
$\partial N(V)-W$.  Thus we can cap off $F_1$ in $\partial N(V)-W$ to
obtain a sphere.  Let $E_1$ denote the closure of the component of the
complement of this sphere in $M$ which is disjoint from $W$.  The
orbit of $E_1$ under $G$ is a pairwise disjoint collection
$E_1$,\dots, $E_r$ such that each $F_i\subseteq\partial E_i$. Suppose
that $E_1$ is not a ball.  Since $M$ is irreducible, $\cl(M-E_1)$ is a
ball containing $E_2$.  But since $E_2\cong E_1$ is not a ball, this
is impossible.  Thus each $E_i$ must be a ball.

Now the action of $G$ on the orbit $\{ F_1,\dots, F_r\}$ defines a
non-trivial monomorphism $\Phi:\TSG_+(\Gamma,M)\rightarrow S_r$.
Furthermore, since $\TSG_+(\Gamma,M)$ is non-abelian, $r>2$.  Hence the
base surface of the $I$-bundle has at least three boundary components.
We see as follows that $\Phi$ is onto.  Let $(ij)$ be a transposition
in $S_r$.  Then there is a $g\in G$ such that $g(F_i)=F_j$.  We saw
above that each component of $\partial N(V) \cap W$ is setwise
invariant under $G$.  Hence a boundary component of $F_i$ and its
image under the element $g$ are in the same component of $\partial
N(V) \cap W$, and they project to distinct boundary components of the
base surface of the $I$-bundle.  Let $N$ denote a regular neighborhood
in the base surface of these two boundary components together with an
arc between them.  Then $N$ is a disk with two holes.  Now since $M$
is orientable, the pre-image of $N$ in the $I$-bundle is a product
$N\times I$. We add the balls $E_i$ and $E_j$ to $N\times I$ along the annuli $F_i$ and $F_j$ respectively to obtain a solid cylinder $C\times I$.

 We will define a homeomorphism $h:(M,\Gamma)\rightarrow (M,\Gamma)$
as follows.  Let $h|E_i=g|E_i$ and $h|E_j=g^{-1}|E_j$.  Then extend
$h$ within $N\times I$ so that $h$ restricted to the cylinder
$\partial C\times I$ is the identity and $h$ leaves each of the disks
$C\times \{0\}$ and $C\times\{1\}$ setwise invariant.  Next we cap off
the solid cylinder $C\times I$ in $N(V)$ to obtain a ball or pinched
ball $B$ whose boundary intersects $\Gamma$ in either one or two
vertices.  Then extend $h$ within $B$ in such a way that $h|\partial
B$ is the identity and $h$ leaves $\Gamma \cap B$ setwise invariant.
Finally, we extend $h$ to $M-B$ by the identity.  Now by our
construction, $h:(M,\Gamma)\rightarrow (M, \Gamma)$ is an orientation
preserving homeomorphism such that $\Phi(h)=(ij)$.  It follows that
$\Phi$ is onto.  However, this is impossible since $\TSG_+(\Gamma,M)$
is a simple non-abelian group.  Thus the pared manifold $(W,W\cap
(P\cup \Omega))$ cannot be $I$-fibered, and hence must be simple.

  \smallskip

\noindent {\bf Step 4:} We define a group of isometries $K$ of
$W$ and prove $K\cong\TSG_+(\Gamma,M)$.
\smallskip

Since $W$ is simple, it follows from Thurston's Hyperbolization
Theorem for Pared Manifolds \cite{Th} applied to $(W, W\cap (P\cup
\Omega ))$ that $W-(W\cap (P\cup \Omega ))$ admits a finite volume
complete hyperbolic metric with totally geodesic boundary.  Let $D$
denote the double of $W-(W\cap (P\cup \Omega ))$ along its boundary.
Then $D$ is a finite volume hyperbolic manifold, and every element of
$\TSG_{+}(\Gamma,M )$ is induced by an element of $G$ whose
restriction to $W$ can be doubled to obtain a diffeomorphism of $D$.
Now by Mostow's Rigidity Theorem~\cite{Mos}, each such diffeomorphism
of $D$ is homotopic to an orientation preserving finite order isometry
that restricts to an isometry of $W- (W\cap(P\cup \Omega ))$.
Furthermore, the set of all such isometries generates a finite group
$K$ of isometries of $W-(W\cap(P\cup \Omega ))$.  By removing
horocyclic neighborhoods of the cusps of $W-(W\cap(P\cup \Omega ))$,
we obtain a copy of the pair $(W, W\cap (P\cup \Omega ))$ which is
contained in $W-(W\cap(P\cup \Omega))$ and is setwise invariant under
the isometry group $K$.  We shall abuse notation and consider $K$ to
be a finite group of isometries of $(W, W\cap (P\cup \Omega ))$.  Also
$K$ induces a finite group of isometries of the tori and annuli in
$W\cap (P\cup \Omega )$ with respect to a flat metric.  In particular,
$\partial N(V)\cap W$, $\partial'N(E)\cap W$, and $\Omega \cap W$ are
each setwise invariant under $K$.  Finally, it follows from
Waldhausen's Isotopy Theorem~\cite{Wa2} that each element of $K$ is
isotopic to an element of $G$ restricted to $W$ by an isotopy leaving
$W\cap (P\cup \Omega )$ setwise invariant.

We show as follows that $\TSG_+(\Gamma,M)\cong K$.  Let
$a\in\TSG_+(\Gamma,M)$ be induced by the elements $g_1$, $g_2\in G$.  Then, as we
observed in Step 2, $g_1$ and $g_2$ induce the same permutation of the
components of $\partial N(V)\cap W$, $\partial' N(E)\cap W$, and
$\Omega \cap W$.  Now $g_1|W$ and $g_2|W$ are isotopic to some $f_1$,
$f_2\in K$ by isotopies leaving $W\cap (P\cup \Omega )$ setwise
invariant.  Thus $f_1$ and $f_2$ also induce the same permutation of
the components of $\partial N(V)\cap W$, $\partial' N(E)\cap W$, and
$\Omega \cap W$.  Recall that the component of $\partial W$ which
meets $N(V)$ is not a torus.  Thus there is some component $J$ of
$\partial N(V)\cap W$ which has $r\geq 3$ boundary components $\alpha
_{1}, \dots , \alpha _{r}$.  Now $f_1(J)=f_2(J)$ and $f_1(\alpha
_{i})=f_2(\alpha _{i})$ for each $i=1, \dots , r$.  Hence
$f_1^{-1}\circ f_2$ restricts to a finite order diffeomorphism of $J$
which setwise fixes each component of $\partial J$.  Since $J$ is a
sphere with at least three holes, this implies that $f_1^{-1}\circ f_2
\vert J$ is the identity.  Finally, since $f_1$ and $f_2$ are
isometries of $W$ which are identical on the surface $J\subseteq
\partial W$, it follows that $f_1=f_2$.  Thus the automorphism $a\in
\TSG_+(\Gamma,M)$ determines a unique element of $K$, and hence there
is a well-defined homomorphism $\Phi :\TSG_+(\Gamma, M) \rightarrow
K$.  Since every element of $K$ came from such an element of
$\TSG_+(\Gamma, M)$, $\Phi$ is onto.  Now since $ \TSG_+(\Gamma,M)$ is
simple, it follows that $\TSG_+(\Gamma,M)\cong K$.

  \smallskip

\noindent {\bf Step 5:} We extend $K$ to a set $W_2$ whose boundaries
are spheres that do not bound balls in $M-W_2$ and tori that are not
compressible in $M-W_2$.
\smallskip

  Every annulus in $W\cap (P\cup \Omega )$ separates $X$ into two
components.  It follows that for any $\partial N(v)$ which meets $W$,
each component of $\partial N(v)-W$ is either a disk or an annulus.
Let $V_1$ denote the set of vertices of $\Gamma$ such that each
component of $N(V_1)$ meets $W$ and let $E_1$ denote the set of edges
such that each component of $N(E_1)$ meets $W$.  We extend $K$ to
$\partial N(V_1)-W$ as follows.  Extend each element of $K$ radially
within the disks components of $\partial N(V_1)-W$. Next consider an
annulus component $A$ of $\partial N(V_1)-W$.  The boundaries of $A$
must also be the boundary components of an annulus $A'$ in $W\cap
(P\cup \Omega)$. Since $K$ restricts to a finite group of isometries
of $A'$, we can extend $K$ to a finite group of isometries of $A$.  In
this way we have extended $K$ so that it is defined on each sphere in
$\partial N(V_1)$.  Now we extend $K$ radially within each of the
balls in $N(V_1)$ and in $N(E_1)$.  Thus we have extended $K$ to a
finite group $K_1\cong K$ acting faithfully on $W_1=W\cup N(V_1)\cup
N(E_1)$.
  
  The boundary of $W_1$ consists of spheres and tori made up of the
union of annuli in $\Omega\cap W$ with disks and annuli in $\partial
N(V)-W$, together with the tori components of $\partial X\cap W$.  Let
$\{T_1,\dots, T_q\}$ denote the tori components of $\partial W_1$
which are compressible in $M-W_1$.  Since the set $\{T_1,\dots, T_q\}$
is setwise invariant under $G$, the set $\{T_1,\dots, T_q\}$ must be
setwise invariant under $K_1$ as well.  Now we can choose a set of
pairwise disjoint compressing disks $\{D_1,\dots, D_r\}$ for
$\{T_1,\dots, T_q\}$ whose boundaries are setwise invariant under
$K_1$.  Note that depending on the action of $K$ on each $T_i$, we may
have $r>q$.  We add a product neighborhood of each $D_i$ to $W_1$ to
obtain a manifold whose boundary contains more spheres than $\partial
W_1$ but contains no tori which are compressible in $M-W_1$.  We
extend $K_1$ to these product neighborhoods by defining it radially
within each parallel disk.  Furthermore, for any boundary component of
the union of $W_1$ together with these product neighborhoods which
bounds a ball in $M-W_1$, we add that ball and extend $K_1$ radially
within the ball.  Thus we have extended $K_1$ to a finite group
$K_2\cong K_1$ acting faithfully on the manifold $W_2$ consisting of
$W_1$ together with these product neighborhoods and balls. Observe
that all of the components of $\partial W_2$ are either spheres which
do not bound a ball in $M-W_2$ or tori which are not compressible in
$M-W_2$.

 \smallskip

\noindent {\bf Step 6:} We prove $\order(\mathrm{TSG}_+(\Gamma,M))=\order(K)\leq d$ by considering
3 cases.

  \smallskip
  
\noindent {\bf Case 1:} Some component of $\partial W_2$ is a sphere $S$.  
\smallskip

Since $M$ is irreducible and $S$ does not bound a ball in $M-W_2$, $S$
must bound a ball $B$ containing $W_2$.  Now any other sphere in
$\partial W_2$ separates $M$ such that the component of the complement
which is disjoint from $W_2$ is contained in $B$.  Since $B$ is a
ball, this component is also a ball.  As this is contrary to our
definition of $W_2$, all of the components of $\partial W_2$ other
than $S$ must be tori.  By gluing another ball to $ B$ we obtain
$S^3$ such that each of the tori in $\partial W_2$ bounds a (possibly
trivial) knot complement in $S^3$ that is disjoint from $W_2$.  We extend $K_2$ radially within the complementary ball.  Then we
replace each knot complement by a solid torus in such a way we
obtain a homology sphere and we can extend $K_2$ radially within the
solid tori.  In this way we get an isomorphic finite simple non-abelian
group $K_3$ of orientation preserving diffeomorphisms of a homology
sphere.  Now it follows from Zimmerman \cite{Zi} that $K_3\cong A_5$.
Thus $\order(K)=\order( K_3)=60\leq d$.

 \smallskip
  
\noindent {\bf Case 2:} $\partial W_2$ has torus components but no spherical components.

\smallskip

 Recall that every component of $\partial W_2$ is incompressible in
 $M-W_2$.  Suppose that $T$ is a torus component of $\partial W_2$
 which is compressible in $W_2$.  Let $N$ be a product neighborhood of
 a compressing disk in $W_2$. By cutting $T$ along $\partial N\cap T$ and
 capping off with the two disks in $\partial N-T$, we obtain a sphere
 $S$.  Since $T$ is incompressible in $M-W_2$, the component of $M-S$
 which is disjoint from $W_2-N$ is not a ball.  Thus the component of $M-S$ containing $W_2-N$ must be a
 ball.  Hence the union of this ball and the neighborhood $N$ is a
 solid torus $V$ such that $\partial V=T$ and $W_2\subseteq V\subseteq
 M$.
 
 Let $T_1$, \dots, $T_r$ denote the boundary components of $W_2$, and
 suppose that every $T_i$ is compressible in $W_2$. Then, by the
 argument above, each $T_i$ bounds a solid torus $V_i$ such that
 $W_2\subseteq V_i\subseteq M$.  Now $G$ induces an orientation
 preserving finite action on the solid tori $V_1$, \dots, 
 $V_r$ taking meridians to meridians, up to isotopy.  Since $G$
 induces a finite action on the set of tori $\{T_1, \dots, T_r\}$ on
 the level of homology, this means there is also a set of longitudes
 $\{\ell_1, \dots, \ell_r\}$ which are setwise invariant under $G$ up
 to isotopy.  Hence the action that $K_2$ induces on the set of tori
 $\{T_1, \dots, T_r\}$ leaves the set of longitudes $\{\ell_1,
 \dots, \ell_r\}$ setwise invariant up to isotopy.

We obtain a homology sphere $W_3=W_2\cup U_1\cup\dots\cup U_r$ by
gluing a solid torus $U_i$ along each boundary component $T_i$ of
$W_2$ so that a meridian $\mu_i$ of $U_i$ is glued to the longitude
$\ell_i$.  Thus $K_2$ leaves the set of meridians $\{\mu_1, \dots,
\mu_r\}$ setwise invariant up to isotopy.  Now since the action of
$K_2$ on on the set of tori $\{\partial U_1, \dots, \partial U_r\}$ is
finite, for some $q\geq r$, we can find a set of
pairwise disjoint meridians $\{m_1,\dots, m_q\}$ for the solid tori $\{U_1, \dots,U_r\}$ which is setwise
invariant under $K_2$.  Now extend $K_2$ radially from the set of
meridians $\{m_1,\dots, m_q\}$ to a set of pairwise disjoint
meridional disks for the solid tori $U_1$, \dots, $U_r$.  These
meridional disks cut the solid tori $U_1$, \dots, $U_r$ into a set of
solid cylinders, and hence we can also extend $K_2$ radially within
this set of solid cylinders.  In this way we obtain a finite group
$K_3\cong K_2$ acting faithfully on the homology sphere $W_3$.  Thus
it again follows from Zimmerman \cite{Zi} that $K_3\cong A_5$, and
hence $\order(K)=\order( K_3)=60\leq d$.

Therefore, we can assume that some component of $\partial W_2$ is an
incompressible torus in $M$.  If a torus component of $\partial W$ is
setwise fixed by $K_2$, then $K_2$ restricts to a faithful action of
the torus, which is impossible since $K_2$ is a finite simple
non-abelian group.  Thus some incompressible
boundary component of $\partial W_2$ has non-trivial orbit
$\{T_1,\dots, T_q\}$ under $K_2$.  Now either each $T_i$ is isotopic
to a torus in the characteristic family $\Theta $ or each $T_i$ is vertical in a closed up
Seifert fibered component of $M-\Theta $.  Suppose that each $T_i$ is
isotopic to a torus in $\Theta $.  Then without loss of generality we
can assume that each $T_i$ is in $\Theta $.  It follows that there is
a non-trivial monomorphism from $K_2$ to $S_m$ (where $m$ is the
number of tori in $\Theta $).  Hence $\order(K)=\order(K_2)\leq m!\leq
d$.

Thus we can assume that each $T_i$ is a vertical torus in a closed up
Seifert fibered component of $M-\Theta $ and $T_i$ is not isotopic to a torus in $\Theta$.  Let $C_1$, \dots, $C_r$
denote all of the closed up Seifert fibered components of $M-\Theta
$. The action of $K_2$ on the set $\{T_1,\dots, T_q\}$ induces an
action on the set $\{C_1\cap W_2,\dots, C_r\cap W_2\}$, which in turn
defines a homomorphism from $K_2$ to $S_r$.  Since $K_2$ is simple,
either this homomorphism is trivial or $\order(K)=\order(K_2)\leq
r!\leq m!\leq d$.  

Therefore, we can assume that the homomorphism is
trivial, and hence the orbit $\{T_1,\dots, T_q\}$ is contained
in a single Seifert fibered component $C$.  Now there is a non-trivial
monomorphism from $K_2$ to $S_{q}$.  Since $K_2$ is simple and non-abelian, $q\geq 5$.  In
particular, the Seifert fibered space $C\cap W_2$ is a Haken manifold
with more that two boundary components.  Now it follows from
Waldhausen \cite{Wa1} that the fibration on $C\cap W_2$ is unique up
to isotopy.  Hence by Meeks and Scott \cite{MS}, $C\cap W_2$ has a
$K_2$-invariant fibration.  Thus $K_2$ induces an action on the base
surface of the fibration of $C\cap W_2$.  Let $F$ be the base surface
for the Seifert fibration of $C$, and let $F'$ be the base surface for
the Seifert fibration of $C\cap W_2$.  Since the vertical tori $T_i$
are incompressible in $C$, none of the boundary components of $F'$
bounds a disk in $F$.  Furthermore, since there are at least five such
tori, $\chi(F)\leq \chi(F') <0$.  Note that since the action of $K_2$
on the the set $\{T_1,\dots, T_q\}$ is non-trivial, $K_2$ cannot take each fiber of $C\cap
W_2$ to itself. Thus $K_2$ induces an isomorphic action on $F'$.  Now
using Hurwitz \cite{Hu}, we have $\order(K)=\order(K_2) \leq
84|\chi(F')|\leq 84 |\chi(F)|\leq d$.
 
\smallskip
  
\noindent {\bf Case 3:}  $\partial W_2$ is empty.
\smallskip

In this case, $W_2$ is the 3-manifold $M$.  Suppose
that the set of characteristic tori $\Theta$ is non-empty.  Since
$\Theta $ is unique up to isotopy, we can find an isotopic set of tori
$\Theta'$ which are setwise invariant under $K_2$.  Thus there is a
homomorphism from $K_2$ to $S_m$ (recall that $m$ is the number of tori in $\Theta$).  Furthermore, since $K_2$ is a
non-trivial finite simple non-abelian group of orientation preserving
diffeomorphisms of $W_2$, $K_2$ cannot leave a torus setwise
invariant.  Thus the homomorphism is injective, and hence
$\order(K)=\order(K_2)\leq m!\leq d$.

Therefore we can assume that $\Theta $ is empty.  If $M$ is not Seifert
fibered, then $\order(K)=\order(K_2)\leq d$ by Case 2 of Step 1.  So
we can further assume that $M$ is Seifert fibered.  Now by Waldhausen
\cite{Wa1}, if $M$ is a closed Haken manifold other than the 3-torus
and the double of the twisted $I$-bundle over a Klein bottle, then $M$
has a fibration which is unique up to isotopy.  Also, by
Ohshika \cite{Oh}, if $M$ is a non-Haken manifold with infinite
$\pi_1$, then $M$ has a fibration which is unique up to isotopy.
Since $M$ is irreducible, if the fibration is unique up to isotopy,
then we can apply Meeks and Scott \cite{MS} to get a $K_2$-invariant
fibration.  In this case, we can argue as in the end of Case 2 to again conclude that $\order(K)\leq d$.

Thus we an assume that either $M$ is the 3-torus, $M$ is the double of
the twisted $I$-bundle over a Klein bottle, or $M$ has finite
fundamental group.  Observe that the 3-torus and the twisted
$I$-bundle over a Klein bottle both have flat geometry.  By Meeks and
Scott \cite{MS} any smooth finite group action of a flat manifold
preserves the geometric structure.  However, by considering the lift
of the action to $\mathbb{R}^3$ we see that no finite simple
non-abelian group can act geometrically and faithfully on either the
3-torus or the twisted $I$-bundle over a Klein bottle.  Thus $M$
cannot be either of these manifolds.

Finally, suppose that $M$ has finite fundamental group.  Now by the
proof of the Elliptization Conjecture \cite{MT}, $M$ has
elliptical geometry, and hence by Dinkelbach and Leeb \cite{DL} we can
assume that $K_2$ acts geometrically on $M$.  However by the
classification of orientation preserving isometry groups of elliptic
3-manifolds of Kalliongis and Miller \cite{KM} and McCullough
\cite{Mc}, no finite simple non-abelian geometric group action of an
elliptic 3-manifold has order greater than 60.  Thus $\order(K_2)\leq
60\leq d$. \end{proof}

\begin{proof}[Proof of Theorem \ref{mainthm}] 
Let $d$ be the number given by Proposition \ref{mainprop} for the
manifold $M$ and choose $n>d$.  Then the alternating group $A_n$ is a
non-abelian simple group, and by Proposition \ref{mainprop} no
embedding of any graph $\Gamma$ in $M$ has $\TSG_+(\Gamma,M)\cong A_n$.  Now
since $\TSG_+(\Gamma,M)$ is either equal to $\TSG(\Gamma,M)$ or is a
normal subgroup of $\TSG(\Gamma,M)$ of index 2, there is no embedding
of a graph $\Gamma$ in $M$ such that $\TSG(\Gamma,M)\cong A_n$.\end{proof}

\section{Proofs of Theorems \ref{univthm} and \ref{rmk1}}

 In contrast with Theorem 1, we prove in Theorem 2 that if the
manifold $M$ can vary (even among the hyperbolic rational homology
spheres) then the collection of topological symmetry groups of
embedded graphs is universal.  We begin with the following
proposition.
 
\begin{prop}
Let $M$ be a connected 3-manifold, and let $G$ be a finite group of
diffeomorphisms acting freely on $M$. Then there is a graph $\Lambda$ graph
embedded in $M$ such that $\TSG(\Lambda,M) \cong G$.
\end{prop}

\begin{proof} 
Let $n=\order(G)$.  If $n=1$ or $2$, we can choose $\Lambda$ to be a single vertex or a single edge in $M$,
respectively.  Thus we assume that $n>2$.
 
 Let $U$, $V$, and $W$ be sets of $n$ vertices each.  Let
$\gamma$ be the abstract graph with vertices in $U\cup V\cup
W$ and an edge between a pair of vertices if and only if
precisely one of the vertices is in $V$.  Then every automorphism of
$\gamma$ leaves the set $V$ setwise invariant since the valence of the
vertices in $V$ is twice that of the vertices in $U\cup W$.  It
follows that if an automorphism of $\gamma$ setwise fixes an edge,
then it fixes both vertices of that edge.

 We embed $\gamma$ in $M$ as follows.  Let $u$, $v$, and $w$ be points
in $M$ whose orbits under $G$ are disjoint.  Embed the sets $U$, $V$, and $W$ as the orbits of $u$, $v$, and $w$
respectively under $G$.  We abuse notation and refer to both the
abstract and embedded sets of vertices as $U$, $V$, and $W$.
Since $G$ acts freely on $M$, $G$ induces a faithful action of the
abstract graph $\gamma$ such that no non-trivial element of $G$ fixes
any vertex or inverts any edge of $\gamma$.  Furthermore, the quotient map
$\pi:M\to M/G$ is a covering map and $M/G$ is a 3-manifold.
 
 Let $\{\varepsilon_1,\dots, \varepsilon_m\}$ consist of one
representative from each orbit of the edges of the abstract graph
$\gamma$ under $G$, and for each $i$ let $x_i$ and $y_i$ denote the
embedded vertices of $\varepsilon_i$.  Since $M$ is path connected,
for each $i$ we can choose a path $\alpha_i$ in $M$ from $x_i$ to
$y_i$ and let $\alpha_i'=\pi(\alpha_i)$.  Since $G$ leaves each of
$U$, $V$, and $W$ setwise invariant, each $\alpha_i'$ has distinct
endpoints.  Now, by general position in $M/G$, we can homotop each
$\alpha_i'$ fixing its endpoints to a simple path $\rho_i'$ such that
the interiors of the $\rho_i'$ are pairwise disjoint and are disjoint
from $\pi(V\cup U\cup W)$.  For each $i$, let $\rho_i$ denote the lift
of the path $\rho_i'$ beginning at $x_i$.  Then $\rho_i$ is a simple
path in $M$, and since $\rho_i'$ is homotopic fixing its endpoints to
$\alpha_i'$, the other endpoint of $\rho_i$ is $y_i$.  For each $i$,
embed the abstract edge $\varepsilon_i$ as the image of $\rho_i$ in
$M$.
 
 Now let $\varepsilon$ be an arbitrary edge of $\gamma$.  Since no
edge of $\gamma$ is setwise fixed by a non-trivial element of $G$, there is a
unique $g\in G$ and $i$ such that $\varepsilon=g(\varepsilon_i)$.
Hence we can unambiguously embed $\varepsilon$ as $g(\rho_i)$.  Let
$\Gamma$ consist of the embedded vertices $V\cup U\cup W$ together
with embeddings of the edges of $\gamma$ defined in this way.  It
follows from our choice of the paths $\rho_i'$ in $M/G$ that these
embedded edges are pairwise disjoint and their interiors are disjoint
from the set of vertices $V\cup U\cup W$.  Thus $\Gamma$ is indeed an
embedding of $\gamma$ in $M$, and is setwise invariant under $G$.
 
 Now let the set $\{\rho_1,\dots, \rho_m\}$ consist of one representative
from each orbit of the embedded edges of $\Gamma$ under $G$.  We
create a new embedded graph $\Lambda$ by adding $i$ vertices of
valence 2 to the interior of every edge in the orbit of $\rho_i$ in such
a way that $G$ leaves $\Lambda$ setwise invariant.  Then $G$ induces a
faithful action on $\Lambda$, and hence is isomorphic to a subgroup of
$\TSG(\Lambda,M)$.
  
We prove as follows that $G\cong \TSG(\Lambda,M)$.  Let $h$ be a
homeomorphism of $M$ inducing a non-trivial automorphism of $\Lambda$.
Since $\Gamma$ has no vertices of valence 2, $h$ leaves $\Gamma$
setwise invariant inducing a non-trivial automorphism of $\Gamma$.
Hence there is some edge $e$ of $\Gamma$ such that $h(e)\not=e$.  Now
$e$ is in the orbit of some $\rho_i$ under $G$, and hence as a path in
$\Lambda$, $e$ contains precisely $i$ vertices of valence 2.  Thus
$h(e)$ also contains precisely $i$ vertices of valence 2, and hence is
also in the orbit of $\rho_i$ under $G$.  It follows that for some $g\in
G$, $g(e)=h(e)$.  Now $f=g^{-1}h$ is a homeomorphism of $(M, \Gamma)$
taking $e$ to itself, and hence fixing both vertices of $e$ as an edge
in $\Gamma$.
 
 Suppose, for the sake of contradiction, that there is some edge $e'$ adjacent to $e$ such that $f(e')\not=e'$.  Since $f$ fixes
both vertices of $e$, $f$ must fix the vertex $x=e\cap e'$.  By
repeating the above argument with $f(e')$ instead of $h(e)$, we see
that there is a $g_1\in G$ such that $g_1(e')=f(e')$.  However,
since $G$ acts freely on $M$, $g_1$ cannot fix $x$.  Thus $g_1(e')$ has vertices $x$ and $g_1(x)$.  Since $g_1(e')$ is an edge of
$\Gamma$, precisely one of its vertices is contained in $V$.
But this is impossible since $g_1$ leaves $V$ setwise invariant.
Thus $f(e')=e'$, and hence inductively we see that $f$ leaves every
edge of $\Gamma$ setwise invariant.  Since $f$ cannot interchange the
vertices of any edge of $\Gamma$, $f$ induces the trivial automorphism
on $\Gamma$ and hence on $\Lambda$ as well.  Thus, $g$ induces the
same automorphism as $h$ on $\Lambda$.  It follows that
$\TSG(\Lambda,M) \cong G$.\end{proof}
\smallskip

Now Theorem \ref{univthm} follows immediately from Proposition~2, since
Cooper and Long \cite{CL} have shown that for every finite group $H$,
there is a hyperbolic rational homology 3-sphere $M$ with a group of
diffeomorphisms $G\cong H$ such that $G$ acts freely on $M$.  

\smallskip
  
 It was proved in \cite{FNPT} that if a 3-connected graph $\Gamma$ is
embedded in $S^3$, then $\TSG_+(\Gamma,S^3)$ is isomorphic to a
subgroup of the group of orientation preserving diffeomorphisms
$\Diff_+(S^3)$.  We show below that this is not true for all 3-manifolds.
 
\begin{proof}[Proof of Theorem \ref{rmk1}]  
By the Geometrization Theorem \cite{MF, MT, MT2}, $M$ can be decomposed into geometric pieces. Also, since $M$ is irreducible and not Seifert fibered, $M$
does not admit a circle action.  Furthermore, by the Elliptization
Theorem \cite{MT}, a 3-manifold with finite fundamental group is
elliptic and hence Seifert fibered.  Thus $M$ has infinite $\pi_1$,
has no circle action, and is irreducible.  Hence by Kojima \cite{Ko}
there is a bound on the order of finite groups of diffeomorphisms of
$M$.  In particular, there is a prime $p>3$ such that $\mathbb{Z}_p$
is not a subgroup of $\Diff_+(M)$.  Now it follows from \cite{Fl} that
there is an embedding $\Delta$ of the complete graph $K_p$ in the
interior of a ball $B$ such that $(B,\Delta)$ has an orientation
preserving diffeomorphism $h$ which induces an automorphism of order
$p$ on $\Delta$.  Since $h$ is orientation preserving $h$ is isotopic
to the identity on $B$.  Thus we can modify $h$ by an isotopy to get a
diffeomorphism $g$ of $(B,\Delta)$ such that $g|\partial B$ is the
identity and $g$ induces an automorphism of order $p$ on $\Delta$.
Now we embed $B$ in $M$, and extend $g$ to $M$ by the identity.  This
gives us an embedding $\Gamma$ of $K_p$ in $M$ with $\mathbb{Z}_p \leq
\TSG_+(\Gamma)$.  It follows that $\TSG_+(\Gamma)$ cannot be
isomorphic to any finite subgroup of $\Diff_+(M)$.  Finally, since $p>3$, $K_p$ is 3-connected.
\end{proof}

  \end{document}